\documentclass[11pt,a4paper]{article}
\usepackage[T1]{fontenc}
\usepackage{lmodern,amsmath,amssymb,amsthm,mathtools,microtype}
\usepackage[margin=28mm]{geometry}
\usepackage{xcolor}
\usepackage[colorlinks=true,linkcolor=blue!45!black,urlcolor=blue!45!black,citecolor=blue!45!black]{hyperref}
\usepackage{fancyhdr}
\numberwithin{equation}{section}
\newtheorem{theorem}{Theorem}
\newtheorem{lemma}{Lemma}

\theoremstyle{remark}
\newtheorem*{remark}{Remark}
\newcommand{\e}{\mathrm{e}}
\newcommand{\R}{\mathbb R}
\newcommand{\Z}{\mathbb Z}
\newcommand{\M}{\mathfrak M}
\newcommand{\m}{\mathfrak m}
\newcommand{\SSS}{\mathfrak S}
\newcommand{\ph}{\varphi}
\newcommand{\eps}{\varepsilon}
\newcommand{\prim}{\sideset{}{^{*}}\sum}
\DeclareMathOperator{\Ree}{Re}

\title{A prime plus a square\\[3pt]
A conditional $11/20+\varepsilon$ exceptional-set bound}
\author{J\'anos Pintz\thanks{Keywords: prime plus square; exceptional numbers.\newline
2020 Mathematics Subject Classification: Primary 11P32; Secondary 11P55, 11M26.\newline
The first author was supported by the National Research Development
and Innovation Office of Hungary, NKFIH, K133819, K147153,
Excellence 151341.} \and Ingo Alth\"ofer}
\date{12 October 2026}
\hypersetup{pdftitle={A prime plus a square: a conditional 11/20 exceptional-set bound}}
\begin{document}
\maketitle
\thispagestyle{plain}
\begin{abstract}
Let $E_\square(N)$ count the positive integers at most $N$ that cannot
be expressed as a prime plus a nonnegative integer square.
Assuming that every nontrivial zero of every primitive Dirichlet
$L$-function has real part at most $7/8$, we prove
$E_\square(N)\ll_\eps N^{11/20+\eps}$ for every $\eps>0$.
The proof combines a conductor-weighted mean-square estimate with
a uniform estimate for certain families having a prescribed quadratic
component. For these families, the large sieve depends on the varying
part of the conductor, while an approximate functional equation controls
the fourth moment using the full conductor. A zero-detection argument
then absorbs all bands of zeros with real part at least $129/167$.
The remaining bands and the minor arcs give the exponent $11/20$.
The proof includes even denominators, completion and positivity of the
singular series, and removal of proper prime powers.
\end{abstract}

\section{Statement and scope}

Let
\[
 E_\square(N)=\#\{n\in\Z:1\le n\le N,\quad
 n\ne p+k^2\text{ for every prime }p\text{ and integer }k\ge0\}.
\]
The value $k=0$ is allowed. Our hypothesis is
\begin{equation}\label{hyp}
 H_{7/8}:\qquad
 \Ree\rho\le\frac78
 \quad\text{for every nontrivial zero of every primitive Dirichlet }
 L(s,\chi).
\end{equation}
The primitive character of conductor $1$ is included.
\begin{theorem}\label{thm:main}
Assume \eqref{hyp}. For every $\eps>0$,
\begin{equation}\label{eq:theorem}
 E_\square(N)\ll_\eps N^{11/20+\eps}\qquad(N\ge2).
\end{equation}
\end{theorem}

This is a conditional theorem: a zero bound for the Riemann zeta
function alone does not supply \eqref{hyp}. Numerically,
$11/20=0.55$, a decrease of $1/180$ from $5/9$.
The statement retains an arbitrarily small positive loss in the
exponent; it does not assert the bound with exponent exactly $11/20$.

We write $\e(t)=\exp(2\pi it)$. All zeros are counted with multiplicity,
and starred character sums run over primitive characters. Symbols
$\nu,\omega>0$ denote arbitrarily small losses and may be decreased
between estimates. Implied constants may depend on these losses
and on fixed smooth weights. We prove the bound on $X<n\le2X$
and then sum dyadically.
\section{Density estimates and large-sieve inputs}

For primitive $\psi$, let $N(\sigma,H,\psi)$ count its nontrivial
zeros $\rho=\beta+i\gamma$ with $\beta\ge\sigma$ and $|\gamma|\le H$,
with multiplicity, and put
\[
 N_2(\sigma,R,H)=\sum_{r\le R}\prim_{\psi\bmod r}
                  N(\sigma,H,\psi).
\]
We use the Ingham estimate in \cite[Theorem 28.3, (28.11)]{MV}:
\begin{equation}\label{eq:ingham}
 N_2(\sigma,R,H)\ll
 (R^2H)^{3(1-\sigma)/(2-\sigma)}\log^6(2RH)
 \quad(1/2\le\sigma\le1;\ R,H\ge1).
\end{equation}
We also use two estimates of Heath-Brown. The first is
\cite[Theorem 2, equation (3)]{HBD}:
\begin{equation}\label{eq:HB-density}
 N_2(\sigma,R,H)\ll_\xi
 \bigl(R^2H^{6/5}\bigr)^{(20/9)(1-\sigma)+\xi}
 \quad(1/2\le\sigma\le1;\ R,H\ge1;\ \xi>0).
\end{equation}
The second is \cite[Theorem 3]{HBD}:
\begin{equation}\label{eq:HB-high}
 N_2(\sigma,R,H)\ll_\xi
 (R^2H^2)^{2(1-\sigma)+\xi}
 \quad(\sigma_0\le\sigma\le1),\qquad
 \sigma_0=\frac{129}{167}.
\end{equation}
After summation against a rapidly decreasing height weight, the last
two bounds give conductor exponents $(40/9)(1-\sigma)$ and
$4(1-\sigma)$, respectively.

For arbitrary complex coefficients, the multiplicative large sieve gives
\begin{equation}\label{eq:ordinary-LS}
 \sum_{t\le T}\sum_{\psi\bmod t}^{*}
 \left|\sum_{m\le M}c_m\psi(m)\right|^2
 \ll (T^2+M)\sum_{m\le M}|c_m|^2.
\end{equation}
Its hybrid version, integrated over a height interval of length
$O(H)$, has $T^2H+M$ in place of $T^2+M$, up to a harmless
logarithmic factor. We use the standard versions for finite
Dirichlet polynomials, including well-spaced evaluation points;
see \cite[Chapters 7 and 12]{Montgomery} and the zero-detection
formulation in \cite[Section 2]{CZ}.

The quadratic large sieve \cite[Theorem 1]{HB} states that
\begin{equation}\label{eq:HB}
 \sideset{}{^{\flat}}\sum_{d\le M}
 \left|\sideset{}{^{\flat}}\sum_{q\le U}b_q
       \left(\frac qd\right)\right|^2
 \ll_\omega (MU)^\omega(M+U)
       \sideset{}{^{\flat}}\sum_{q\le U}|b_q|^2
 \qquad(\omega>0).
\end{equation}
Both flat-marked sums range over positive odd squarefree integers;
the symbol is the Jacobi symbol. All the inputs in this section
are unconditional.

\section{Smooth weights and major arcs}

Fix nonnegative functions
\[
 v\in C_c^\infty((1/4,1/2)),\qquad
 w\in C_c^\infty((1/2,3)),
\]
with $v\not\equiv0$ and $w=1$ on $[3/4,2]$. Set $Y=\sqrt X$ and
\begin{align*}
 S(\alpha)&=\sum_{m\ge1}\Lambda(m)w(m/X)\e(m\alpha),\\
 \Theta(\alpha)&=\sum_{k\ge0}v(k/Y)\e(k^2\alpha),\\
 R_X(n)&=\int_0^1S(\alpha)\Theta(\alpha)\e(-n\alpha)\,d\alpha\\
 &=\sum_{k\ge0}v(k/Y)w((n-k^2)/X)\Lambda(n-k^2).
\end{align*}
We set $\Lambda(m)=0$ for $m\le0$. In particular $R_X(n)\ge0$.
Put
\[
 V(\beta)=\int_0^\infty w(t/X)\e(\beta t)\,dt,
 \qquad W(\beta)=\int_\R v(t/Y)\e(\beta t^2)\,dt.
\]
Since $t\asymp Y$ on the support in $W$, repeated integration by
parts gives, for every fixed $A>0$,
\begin{equation}\label{eq:VW}
 |V(\beta)|\ll_A X(1+X|\beta|)^{-A},\qquad
 |W(\beta)|\ll_A Y(1+X|\beta|)^{-A}.
\end{equation}
For $X<n\le2X$, Fourier inversion gives
\begin{equation}\label{eq:J}
 \int_\R V(\beta)W(\beta)\e(-n\beta)\,d\beta
 =\int_\R v(t/Y)w((n-t^2)/X)\,dt
 =\kappa\sqrt X,
 \qquad\kappa=\int_\R v(u)\,du>0.
\end{equation}
Indeed, $t^2<X/4$ on the support, so $(n-t^2)/X\in[3/4,2]$.

Choose
\begin{equation}\label{eq:PQ}
 P=X^{9/20},\qquad \eta=\frac1{100},\qquad
 Q=\left\lfloor\frac{X}{PX^\eta}\right\rfloor.
\end{equation}
For large $X$, $P=o(Y)$, $Q/Y\to\infty$, and $Q>2P$.
The major arcs $\M\subset\R/\Z$ are the disjoint arcs
\[
 \alpha=a/q+\beta,\qquad q\le P,\quad(a,q)=1,
 \quad |\beta|\le1/(qQ).
\]
For $q=1$ the single reduced residue is interpreted modulo $1$.
Let $\m$ be the complement. Dirichlet approximation gives, for
$\alpha\in\m$, a reduced $a/q$ with $P<q\le Q$ and
$|\alpha-a/q|\le1/(qQ)\le1/q^2$.
The quadratic Weyl bound and partial summation imply
\[
 |\Theta(\alpha)|^2
 \ll_\nu Y^{2+\nu}\left(q^{-1}+Y^{-1}+qY^{-2}\right)
 \ll_\nu X^{1+\nu}/P.
\]
A factor $2$ in the linear phase after differencing changes the
reduced denominator by at most a factor $2$, which is harmless here.
By Parseval and $\sum_{m\le3X}\Lambda(m)^2\ll X\log^2X$,
\begin{equation}\label{eq:minor}
 \sum_{n\in\Z}\left|\int_\m S(\alpha)\Theta(\alpha)
                   \e(-n\alpha)\,d\alpha\right|^2
 \ll_\nu X^{2+\nu}/P.
\end{equation}

\section{Complete square sums, including even denominators}

For $(a,q)=1$ define
\[
 G(a,q)=\sum_{b\bmod q}\e(ab^2/q),\qquad g_q(a)=q^{-1}G(a,q).
\]
The elementary identity
\[
 |G(a,q)|^2=
 \sum_{h\bmod q}\e(ah^2/q)
       \sum_{b\bmod q}\e(2ahb/q)
\]
shows that only $h=0$ contributes for odd $q$. For even $q$, only
$h=0,q/2$ contribute; the second contribution is $q(-1)^{q/2}$,
as $a$ is odd. Thus
\begin{equation}\label{eq:gauss}
 |g_q(a)|^2=
 \begin{cases}
 q^{-1},&q\text{ odd},\\
 0,&q\equiv2\pmod4,\\
 2q^{-1},&4\mid q.
 \end{cases}
 \qquad |g_q(a)|\le\sqrt{2/q}.
\end{equation}

Poisson summation in progressions modulo $q$ gives, uniformly on
all the major arcs,
\begin{equation}\label{eq:theta-major}
 \Theta(a/q+\beta)=g_q(a)W(\beta)+O_A(X^{-A})
 \qquad(A>0).
\end{equation}
To check uniformity, a nonzero Poisson mode $h$ has phase derivative
$2\beta t-h/q$, of magnitude $\gg |h|/q$ on the support, since
$|\beta|Y\le Y/(qQ)=o(1/q)$. Repeated integration by parts bounds
its integral by $C_BY(|h|Y/q)^{-B}$. The normalized complete Fourier
coefficients have absolute value at most $1$. Summing the nonzero
modes gives $O_B(Y(q/Y)^B)$, which is $O_A(X^{-A})$ because
$q\le P=X^{9/20}<X^{1/2}$ and $B$ can be arbitrarily large.

\section{A finite partition of the zeros}

For $1/2\le\sigma\le7/8$, use the conductor exponent
\begin{equation}\label{eq:a}
 a(\sigma)=
 \begin{cases}
 6(1-\sigma)/(2-\sigma),&1/2\le\sigma\le13/20,\\
 (40/9)(1-\sigma),&13/20<\sigma<\sigma_0,\\
 4(1-\sigma),&\sigma_0\le\sigma\le7/8.
 \end{cases}
\end{equation}
The first two expressions agree at $13/20$.
At $\sigma_0=129/167$ we switch to the stronger bound
\eqref{eq:HB-high}; thus $a$ has a downward jump there.
In particular,
\begin{equation}\label{eq:a-low}
 a(\sigma)>1\quad(1/2\le\sigma<\sigma_0),\qquad
 \frac12\le a(\sigma)<1\quad(\sigma_0\le\sigma\le7/8).
\end{equation}
The limiting value from the left at $\sigma_0$ is $1520/1503>1$.

\begin{lemma}\label{lem:density}
For every $\nu>0$ and $R\ge1$,
\begin{equation}\label{eq:density}
 \sum_{r\le R}\sum_{\psi\bmod r}^{*}
 \sum_{\substack{\rho=\beta+i\gamma\\\beta\ge\sigma}}
       (1+|\gamma|)^{-10}
 \ll_\nu R^{a(\sigma)+\nu}
 \qquad(1/2\le\sigma\le7/8).
\end{equation}
For zeros with real part less than $1/2$, the same bound holds with
conductor exponent $2$.
\end{lemma}
\begin{proof}
Apply the corresponding estimate \eqref{eq:ingham},
\eqref{eq:HB-density}, or \eqref{eq:HB-high} on dyadic height
blocks, including $|\gamma|\le1$ in the first block.
Choose $\xi$ sufficiently small relative to $\nu$.
All height exponents in these estimates are less than $2$ in
their ranges of use, so the weight of exponent $10$ makes the
height sums converge. Absorb logarithms into $R^\nu$, adjusting
constants for bounded $R$. Ordinary zero counting gives conductor
exponent $2$ for the zeros with real part less than $1/2$.
\end{proof}

Fix small positive $\omega,\delta<1/100$, independent of $X$.
Partition $[1/2,\sigma_0)$ and $[\sigma_0,7/8]$ into finitely many
bands $I$, each of width at most $\omega$, using $13/20$ as a
partition point as well. Assign each shared endpoint to exactly
one band; assign $\sigma_0$ to an upper band.
For a band let $\sigma_I,h_I$ be its lower and upper endpoints,
and put $a_I=a(\sigma_I)$.
Adjoin the group $I_-$ of all zeros with $\beta<1/2$, setting
$h_{I_-}=1/2$ and $a_{I_-}=2$.
Under \eqref{hyp}, these groups exhaust the nontrivial zeros.
Their number is finite and depends only on $\omega$.

Call $I_-$ and the bands below $\sigma_0$ the \emph{lower groups},
and call the bands at or above $\sigma_0$ the \emph{upper groups}.
For a primitive character $\psi$ of conductor $r$, define
\[
 Z_{\psi,I}=\sum_{\rho\in I}(1+|\Im\rho|)^{-10},\qquad
 Z_{r,I}=\sum_{\psi\bmod r}^{*}Z_{\psi,I}.
\]
Here $\rho\in I$ refers to its assigned group of real parts.
Lemma~\ref{lem:density} gives
\begin{equation}\label{eq:band-density}
 \sum_{r\le R}Z_{r,I}\ll_\nu R^{a_I+\nu}.
\end{equation}
Also $Z_{\psi,I}\ll\log(r+2)$ by ordinary zero counting.

\section{The explicit formula, separated by groups}

For primitive $\psi$ of conductor $r$, put
\[
 E_\psi(\beta)=\sum_{m\ge1}\Lambda(m)\psi(m)w(m/X)\e(m\beta)
                         -\delta_\psi V(\beta),
\]
where $\delta_\psi=1$ exactly when $r=1$.
Define
\[
 M_\beta(s)=\int_0^\infty w(t/X)\e(\beta t)t^{s-1}\,dt,
 \qquad E_{\psi,I}(\beta)=-\sum_{\rho\in I}M_\beta(\rho).
\]
For $-1/2\le u\le2$, repeated integration by parts in $\log(t/X)$
gives
\begin{equation}\label{eq:Mellin}
 |M_\beta(u+i\tau)|\ll
 X^u(1+X|\beta|)^{10}(1+|\tau|)^{-10}.
\end{equation}
Mellin inversion against $-L'/L$, followed by a shift from
$\Ree s=2$ to $\Ree s=-1/2$, yields
\begin{equation}\label{eq:explicit}
 E_\psi(\beta)=\sum_I E_{\psi,I}(\beta)
       +O\bigl((1+X|\beta|)^{10}\log(r+2)\bigr).
\end{equation}
The pole at $1$ gives the subtracted term. The only possible
trivial-zero residue crossed is at $0$, of bounded multiplicity.
On the new line the functional equation and the absolutely convergent
logarithmic derivative at $3/2-i\tau$ give
$L'/L\ll\log(r(|\tau|+2))$; together with \eqref{eq:Mellin}, this
bounds the new integral and the trivial-zero contribution as stated.
The shift is first made at admissible finite heights and then passed
to the limit. The zero sums converge absolutely by \eqref{eq:Mellin}.

For each group, taking absolute values directly gives
\begin{equation}\label{eq:band-linear}
 |E_{\psi,I}(\beta)|\ll
 X^{h_I}(1+X|\beta|)^{10}Z_{\psi,I}.
\end{equation}
Cauchy's inequality over the zeros, using $Z_{\psi,I}\ll\log(r+2)$,
also gives
\[
 |E_{\psi,I}(\beta)|^2\ll
 X^{2h_I}(1+X|\beta|)^{20}\log(r+2)Z_{\psi,I}.
\]
Since \eqref{eq:VW} implies
$\int_\R|W(\beta)|^2(1+X|\beta|)^{20}\,d\beta\ll Y^2/X=1$,
we obtain
\begin{equation}\label{eq:band-integral}
 \int_\R|W(\beta)|^2|E_{\psi,I}(\beta)|^2\,d\beta
 \ll X^{2h_I}\log(r+2)Z_{\psi,I}.
\end{equation}

For $q\le P$ and $(a,q)=1$, write
\[
 A_q(a,\beta)=S(a/q+\beta)-\frac{\mu(q)}{\ph(q)}V(\beta),
 \qquad \tau(\chi)=\sum_{b\bmod q}\chi(b)\e(b/q).
\]
If $\chi^*$ is the primitive character inducing $\chi$, define
\[
 A_{q,I}(a,\beta)=\frac1{\ph(q)}\sum_{\chi\bmod q}
       \tau(\overline\chi)\chi(a)E_{\chi^*,I}(\beta).
\]
Character expansion and \eqref{eq:explicit} give
\begin{equation}\label{eq:A-expansion}
 A_q(a,\beta)=\sum_I A_{q,I}(a,\beta)+D_q(a,\beta),
\end{equation}
with
\begin{equation}\label{eq:D-trace}
 \sum_{a\bmod q}^{*}|D_q(a,\beta)|^2
 \ll q\log^C X\,(1+X|\beta|)^{20}
\end{equation}
for an absolute constant $C$.
Here the terms with $(m,q)>1$ and the changes from induced to
primitive characters involve only prime powers with base dividing
$q$; their individual weights are $O(\log^2X)$.
For these terms and the remainder in \eqref{eq:explicit}, character
orthogonality and the exact identity
$\sum_{\chi\bmod q}|\tau(\chi)|^2=\ph(q)^2$
prove \eqref{eq:D-trace}.
For each group, orthogonality also gives the exact identity
\begin{equation}\label{eq:orthog}
 \sum_{a\bmod q}^{*}|A_{q,I}(a,\beta)|^2
 =\frac1{\ph(q)}\sum_{\chi\bmod q}
       |\tau(\chi)|^2|E_{\chi^*,I}(\beta)|^2.
\end{equation}
If $q=rs$ and $\chi$ is induced by primitive $\psi\bmod r$, then
\begin{equation}\label{eq:induced}
 \tau(\chi)=\mu(s)\psi(s)\tau(\psi).
\end{equation}
This vanishes unless $s$ is squarefree and $(r,s)=1$; otherwise
$|\tau(\chi)|=\sqrt r$. All these formulas include even moduli.

\section{A complete character-sum estimate}

The next estimate retains the additive phase in the complete residue
sum. Write $(u,v)=\gcd(u,v)$.
\begin{lemma}\label{lem:complete}
For every positive integer $q$, every Dirichlet character $\chi\bmod q$,
and every integer $n$,
\begin{equation}\label{eq:complete}
 \left|\sum_{b\bmod q}^{*}g_q(b)\chi(b)\e(-nb/q)\right|
 \le2\sqrt{(q,n)}.
\end{equation}
\end{lemma}
\begin{proof}
First, for an arbitrary character $\theta\bmod q$, put
\[
 T_q(\theta,n)=\sum_{b\bmod q}^{*}\theta(b)\e(-nb/q).
\]
Let $d=(q,n)$ and $Q_0=q/d$. Reduction of units modulo $q$ to units
modulo $Q_0$ is surjective. Each fibre has
$\ph(q)/\ph(Q_0)$ elements. The additive character is constant
on each fibre. If $\theta$ is nontrivial on the kernel of this
reduction, its sum on every fibre is zero. Otherwise $\theta$
is induced from a character $\theta_0\bmod Q_0$, and
\[
 T_q(\theta,n)=\frac{\ph(q)}{\ph(Q_0)}
                  T_{Q_0}(\theta_0,n/d).
\]
Since $(n/d,Q_0)=1$, the last sum has the same absolute value as
the ordinary Gauss sum of $\theta_0$. By \eqref{eq:induced} this
is at most $\sqrt{Q_0}$. The case $Q_0=1$ has the same bound
by the usual one-residue convention. Since
$\ph(q)/\ph(Q_0)\le d$, it follows that
\begin{equation}\label{eq:general-gauss}
 |T_q(\theta,n)|\le d\sqrt{Q_0}=\sqrt{q(q,n)}.
\end{equation}

For odd $q$, the quadratic Gauss-sum evaluation gives
\[
 g_q(b)=\epsilon_q q^{-1/2}\left(\frac bq\right),
 \qquad |\epsilon_q|=1\quad((b,q)=1).
\]
This includes $q=1$ by convention. Thus \eqref{eq:general-gauss}
proves \eqref{eq:complete} with constant $1$ when $q$ is odd.
If $q\equiv2\pmod4$, the sum vanishes by \eqref{eq:gauss}.

It remains to handle $4\mid q$. For odd $b$, define
\[
 \chi_4(b)=(-1)^{(b-1)/2},\qquad
 \chi_8(b)=(-1)^{(b^2-1)/8},\qquad
 \chi_{-8}(b)=\chi_4(b)\chi_8(b).
\]
The explicit evaluations at powers of $2$ are
\begin{equation}\label{eq:two-adic}
 g_{2^v}(b)=2^{-v/2}
 \begin{cases}
 1+i\chi_4(b),&v\ge2\text{ even},\\
 \chi_8(b)+i\chi_{-8}(b),&v\ge3\text{ odd}.
 \end{cases}
\end{equation}
For example, the base cases follow by summing the squares modulo
$4$ and $8$. For $v\ge4$, splitting into even and odd arguments
and pairing the odd arguments gives
$G(b,2^v)=2G(b,2^{v-2})$, proving the formulas by induction.
The odd-modulus evaluation and these formulas are also recorded
in \cite[Lemmas 5.1 and 6.1]{MP}.

Write $q=2^v u$ with $u$ odd. The Chinese remainder theorem gives
\[
 G(b,q)=G(bu,2^v)G(b2^v,u).
\]
Combining the odd-modulus formula with \eqref{eq:two-adic} shows
that, on the units modulo $q$,
\[
 g_q(b)=q^{-1/2}\bigl(c_1\eta_1(b)+c_2\eta_2(b)\bigr),
 \qquad |c_1|=|c_2|=1,
\]
for two Dirichlet characters $\eta_1,\eta_2\bmod q$.
The constant factors resulting from the fixed unit arguments $u$
and $2^v$ have absolute value $1$.
Apply \eqref{eq:general-gauss} to $\chi\eta_1$ and $\chi\eta_2$
and divide by $\sqrt q$. This proves \eqref{eq:complete}.
\end{proof}

\begin{lemma}\label{lem:divisor}
For every $\nu>0$, uniformly for integers $X<n\le2X$ and real
$1\le Y_0\le X$,
\begin{equation}\label{eq:divisor}
 \sum_{s\le Y_0}\frac{\sqrt{(s,n)}}s\ll_\nu X^\nu.
\end{equation}
\end{lemma}
\begin{proof}
The inequality $\sqrt{(s,n)}\le\sum_{d\mid(s,n)}\sqrt d$ gives
\[
 \sum_{s\le Y_0}\frac{\sqrt{(s,n)}}s
 \le\sum_{\substack{d\mid n\\d\le Y_0}}\sqrt d
                    \sum_{\substack{s\le Y_0\\d\mid s}}\frac1s
 \le (1+\log Y_0)\sum_{d\mid n}d^{-1/2}.
\]
The last divisor sum is at most the number of divisors of $n$,
which is $\ll_\nu n^\nu$ with any prescribed positive loss.
Decrease that loss to absorb the logarithm and use $n\le2X$.
\end{proof}

\section{The local nonvanishing restriction}

Fix an integer $X<n\le2X$. Call a primitive character $\psi$ of
conductor $r$ \emph{admissible for $n$} if the following conditions hold:
\begin{enumerate}
\item For every odd prime $p\mid(r,n)$, one has $p\Vert r$, and the
local component of $\psi$ modulo $p$ is the Legendre character
$(\cdot/p)$.
\item If $n$ is even, then the exponent of $2$ in $r$ is at most $3$.
\end{enumerate}
These are necessary conditions only; we do not assert that every
admissible character has a nonzero contribution.

\begin{lemma}\label{lem:support}
Let $\psi$ be primitive of conductor $r$, and let $\chi\bmod q$
be the character it induces, where $r\mid q$. If
\[
 \tau(\overline\chi)
 \sum_{b\bmod q}^{*}g_q(b)\chi(b)\e(-nb/q)\ne0,
\]
then $\psi$ is admissible for $n$.
Consequently, writing $d$ for the odd part of $(r,n)$ and $t=r/d$,
one has
\begin{equation}\label{eq:forced-component}
 d\mid n,\quad d\text{ odd and squarefree},\quad(d,t)=1,
 \quad\psi=\lambda_d\psi_t,\qquad
 \lambda_d=(\cdot/d),
\end{equation}
where $\psi_t$ is primitive of conductor $t$, and
\begin{equation}\label{eq:gcd-weight}
 \sqrt{\frac{(r,n)}r}\le\frac{\sqrt8}{\sqrt t}.
\end{equation}
\end{lemma}
\begin{proof}
By \eqref{eq:induced}, nonvanishing of $\tau(\overline\chi)$
requires $q=rs$ with $s$ squarefree and $(r,s)=1$.
If $q\equiv2\pmod4$, the complete square sum vanishes, so this
case cannot occur.

The proof of Lemma~\ref{lem:complete} expresses $g_q(b)$ as
$q^{-1/2}$ times a linear combination of at most two characters
$\eta_j\bmod q$. At an odd prime power $p^v\Vert q$, each $\eta_j$
has local component $(\cdot/p)^v$. Its local component at $2$
has conductor at most $8$. A nonzero complete residue sum requires
at least one of the sums $T_q(\chi\eta_j,n)$ to be nonzero.
The fibre argument in that lemma implies
\begin{equation}\label{eq:factor-condition}
 T_q(\theta,n)\ne0\quad\Longrightarrow\quad
 \theta\text{ factors through modulus }q/(q,n).
\end{equation}

Suppose an odd prime $p$ divides both $r$ and $n$, and write
$p^v\Vert r$. Since $(r,s)=1$, also $p^v\Vert q$.
If $v\ge2$, the local component of $\psi$ is primitive modulo
$p^v$, whereas the twisting character has conductor at most $p$.
Their product is still primitive modulo $p^v$: if it factored
modulo $p^{v-1}$, multiplying by the inverse twist would show that
$\psi$ did too. This contradicts \eqref{eq:factor-condition}, since
$p\mid n$. Therefore $v=1$.
In this case \eqref{eq:factor-condition} forces the twisted local
character to be principal, so the local component of $\psi$ equals
$(\cdot/p)$. This proves the first admissibility condition.

If $n$ is even and $2^v\Vert r$ with $v\ge4$, the same argument
applies: the local component of $\psi$ has conductor $2^v$, while
the twist has conductor at most $8$, which divides $2^{v-1}$.
Their product cannot factor modulo $2^{v-1}$, again contradicting
\eqref{eq:factor-condition}. This proves the second condition.

The odd part $d$ of $(r,n)$ is therefore squarefree and coprime
to $t=r/d$. The components at its primes are fixed by the first
condition. Chinese remainder factorization of a primitive character
then gives \eqref{eq:forced-component}; the component $\psi_t$
is primitive. Here $\lambda_1$ denotes the character of conductor $1$.
Finally, the $2$-part of $(r,n)$ is at most $8$ (and equals $1$
if $n$ is odd). Thus $(r,n)\le8d$, proving \eqref{eq:gcd-weight}.
\end{proof}

\section{Counting the restricted character families}

\begin{lemma}\label{lem:restricted}
Fix a group $I$ of zeros, and abbreviate $a=a_I$.
For $1\le K\le P$ and every $\nu>0$, uniformly for integers
$X<n\le2X$,
\begin{equation}\label{eq:restricted}
 \sum_{r\le K}\sqrt{\frac{(r,n)}r}
 \sum_{\substack{\psi\bmod r\ \mathrm{primitive}\\
                 \psi\ \mathrm{admissible\ for}\ n}}
 Z_{\psi,I}
 \ll_\nu X^\nu K^{3a/4}.
\end{equation}
\end{lemma}
\begin{proof}
For each character in the sum, use the unique decomposition in
\eqref{eq:forced-component}. Fix $d$ and a dyadic block $T\le t<2T$.
By \eqref{eq:gcd-weight}, the contribution of this block is at most
$\sqrt8\,T^{-1/2}V_{d,T}$, where we may enlarge to
\[
 V_{d,T}=
 \sum_{\substack{T\le t<2T\\dt\le K\\(d,t)=1}}
 \sum_{\psi_t\bmod t}^{*}Z_{\lambda_d\psi_t,I}.
\]
This enlargement drops further admissibility conditions only after
the weight has been bounded for the original characters.

There are two independent bounds for $V_{d,T}$.
Since $Z_{\psi,I}\ll\log(r+2)$ and
$\sum_{T\le t<2T}\ph(t)\ll T^2$, ordinary zero counting gives
\begin{equation}\label{eq:restricted-trivial}
 V_{d,T}\ll T^2\log(K+2).
\end{equation}
On the other hand, for fixed $d$, the characters $\lambda_d\psi_t$
are distinct primitive characters of conductors $dt\le K$.
Primitivity follows from $(d,t)=1$ and primitivity of both factors;
the map $\psi_t\mapsto\lambda_d\psi_t$ is injective.
Thus \eqref{eq:band-density} gives, for every small $\xi>0$,
\begin{equation}\label{eq:restricted-density}
 V_{d,T}\ll_\xi K^{a+\xi}.
\end{equation}
Combine \eqref{eq:restricted-trivial} and
\eqref{eq:restricted-density}. Absorbing logarithms and a suitably
small $\xi$ into $X^\nu$, a block contributes at most
\[
 \ll_\nu X^\nu\min\left(T^{3/2},K^aT^{-1/2}\right)
 \le X^\nu K^{3a/4}.
\]
The last inequality follows by splitting at $T=K^{a/2}$:
use the first bound below that threshold and the second above it.
There are $O(\log(2K))$ blocks and at most $\tau(n)$ possible
squarefree odd divisors $d$ of $n$. Since $\tau(n)\ll_\xi X^\xi$
for any $\xi>0$, choosing the auxiliary losses sufficiently small
proves \eqref{eq:restricted} with any prescribed $\nu>0$.
The conductor $1$ is included in the block $d=t=1$.
\end{proof}

\section{A sharper estimate for a fixed quadratic twist}

The improvement needed for the upper groups uses both parts of the
conductor, instead of merely counting the family trivially.
For an odd squarefree $d$, set $\lambda_d=(\cdot/d)$ and define
\[
 \mathcal C_d(T,K)=
 \{\lambda_d\psi_t:T\le t<2T,\ dt\le K,\ (d,t)=1,
                         \ \psi_t\bmod t\text{ primitive}\}.
\]
Every member is primitive of conductor $dt$.
For any primitive character $\chi$, write
\[
 Z_\sigma(\chi)=
 \sum_{\substack{L(\rho,\chi)=0\\0<\Ree\rho<1,\ \Ree\rho\ge\sigma}}
                 (1+|\Im\rho|)^{-10}.
\]

\begin{lemma}\label{lem:twisted-density}
Let $K\ge2$, $1\le T\le K$, $d$ be odd and squarefree, and
$\sigma_0\le\sigma\le7/8$. Put $b=(1-\sigma)/(2-\sigma)$.
For every $\nu>0$, uniformly in $d,T,K$,
\begin{equation}\label{eq:twisted-density}
 \sum_{\chi\in\mathcal C_d(T,K)}Z_\sigma(\chi)
 \ll_\nu K^\nu\bigl((K+T^2)T^4\bigr)^b.
\end{equation}
\end{lemma}
\begin{proof}
We may suppose the family is nonempty. Remove the character of
conductor $1$ during zero detection. Its weighted zero count is
$O(1)$, so it can be restored at the end.

For this family the large-sieve constant is $T^2+M$, with an
absolute change of constant for $t<2T$. Indeed, in
\eqref{eq:ordinary-LS} replace $c_m$ by $c_m\lambda_d(m)$
and then restrict to the allowed $t$. Since $|\lambda_d(m)|\le1$,
this does not increase the squared norm of the coefficients.
Consequently the hybrid constant at height $H\ge2$ is
\begin{equation}\label{eq:twisted-hybrid}
 \Delta_H(M)\ll_\nu (KHM)^\nu(T^2H+M).
\end{equation}
The same observation applies to mollifier coefficients and to the
Dirichlet-polynomial coefficients used in zero detection. Their lengths
are bounded by a fixed power of $KH$, so the factor $(KHM)^\nu$
is absorbed into an arbitrarily small power of $KH$ below.

We next establish the fourth-moment bound
\begin{equation}\label{eq:twisted-fourth}
 \sum_{\chi\in\mathcal C_d(T,K),\ \chi\ne1}
 \int_{-H}^{H}|L(1/2+iu,\chi)|^4\,du
 \ll_\nu (KH)^\nu H(T^2+KH).
\end{equation}
Use the smooth approximate functional equation for
$L(1/2+iu,\chi)^2$. If $r$ is the conductor of $\chi$, the
analytic conductor of this square is $\asymp r^2(1+|u|)^2$;
both Dirichlet polynomials therefore have length
$\ll (KH)^{1+\xi}$, for arbitrarily small fixed $\xi>0$,
with a negligible tail. Their coefficients, before the smooth weights,
are
\[
 \tau(m)\lambda_d(m)\psi_t(m)m^{-1/2-iu},
 \qquad
 \tau(m)\lambda_d(m)\overline{\psi_t(m)}m^{-1/2+iu}.
\]
The factor multiplying the dual polynomial has absolute value $1$.
Separate the two parities and use Mellin inversion for the smooth
weights. On a line of real part $\xi>0$, the conductor-dependent
factors have size at most $(KH)^{O(\xi)}$; the remaining Mellin
integrals are absolutely convergent with uniform rapidly decreasing
kernels. Cauchy's inequality and \eqref{eq:ordinary-LS}, with the
common factor $\lambda_d$ absorbed into the coefficients, apply
to both polynomials. Conjugation preserves the varying character
family. Since
\[
 \sum_{m\le M}\frac{\tau(m)^2}{m}\ll\log^4(2M),
\]
their squared mean over the characters, for each fixed $|u|\le H$,
is $\ll_\nu(KH)^\nu(T^2+KH)$. Taking $\xi$ sufficiently small
and integrating over $u$ proves \eqref{eq:twisted-fourth}.
Only polynomial lengths and arbitrarily small losses are used here;
there is no assumed moment conjecture.

We give the zero-detection calculation to specify the dependence on
both parameters. It is the argument of \cite[Theorem 2.2 and its
proof]{CZ}, with \eqref{eq:twisted-hybrid} as the large-sieve input.
Put
\[
 D=T^2H,\qquad L_0=H(T^2+KH).
\]
For a mollifier length $U_0$ and a smoothing length $Y_0\ge U_0$,
both bounded by a fixed power of $KH$, that argument gives
\begin{align}
 \sum_{\chi\in\mathcal C_d(T,K),\ \chi\ne1}N(\sigma,H,\chi)
 \ll_\nu(KH)^\nu\bigg(&1+
 L_0^{1/3}(D+U_0)^{2/3}Y_0^{2(1-2\sigma)/3}\notag\\
 &+\sup_{u\ge U_0}(D+u)u^{1-2\sigma}e^{-2u/Y_0}\bigg).
 \label{eq:detector}
\end{align}
Here $U_0$ is the parameter denoted by $X$ in that reference;
the supremum harmlessly enlarges its possible dyadic block lengths.
The first nonconstant term controls zeros detected by the integral
containing $L(s,\chi)\sum_{m\le U_0}\mu(m)\chi(m)m^{-s}$;
H\"older's inequality uses the fourth moment of $L$ and the second
moment of this mollifier. The last term controls zeros detected by
the remaining Dirichlet polynomial. The separation of zeros for each
character, their multiplicities, and the enlargement of height by
$O(\log(KH))$ cost only $(KH)^\nu$. This also explains why using
$T$ in the large sieve and $K$ as the conductor bound is legitimate:
the zero-detection proof only needs the displayed uniform moment
and large-sieve inequalities.

Take
\[
 U_0=D,\qquad Y_0=(L_0D^2)^{1/(4-2\sigma)},
\]
with fixed constant factors if needed. Since $L_0\ge D$ and
$4-2\sigma<3$, we have $Y_0\ge D$. The supremum in
\eqref{eq:detector} is at most a constant times
$D^{2-2\sigma}+Y_0^{2-2\sigma}\ll Y_0^{2-2\sigma}$.
The other nonconstant term is also
\[
 L_0^{1/3}D^{2/3}Y_0^{2(1-2\sigma)/3}
       =Y_0^{2-2\sigma}=(L_0D^2)^b.
\]
It follows that the left side of \eqref{eq:detector} is
\[
 \ll_\nu (KH)^\nu
       \bigl(H^3T^4(T^2+KH)\bigr)^b
 \le (KH)^\nu H^{4b}\bigl((K+T^2)T^4\bigr)^b.
\]
Sum this over dyadic height blocks with the weight
$(1+|\gamma|)^{-10}$. The resulting height series converges,
since $4b<1$ in the stated range (and $\nu$ can be made small).
The first block is bounded using $H=2$.
Restore the conductor-$1$ contribution. This proves
\eqref{eq:twisted-density}.
\end{proof}

\begin{lemma}\label{lem:restricted-high}
Let $I$ be an upper group, $\sigma=\sigma_I$, and
\[
 C(\sigma)=\frac{10-11\sigma}{3}.
\]
For $2\le K\le P$, uniformly for integers $X<n\le2X$,
\begin{equation}\label{eq:restricted-high}
 \sum_{r\le K}\sqrt{\frac{(r,n)}r}
 \sum_{\substack{\psi\bmod r\ \mathrm{primitive}\\
                 \psi\ \mathrm{admissible\ for}\ n}}Z_{\psi,I}
 \ll_\nu X^\nu K^{C(\sigma)}.
\end{equation}
\end{lemma}
\begin{proof}
As in Lemma~\ref{lem:restricted}, fix a squarefree odd $d\mid n$
and a dyadic block $T\le t<2T$. Its contribution is bounded by
$\sqrt8\,T^{-1/2}V_{d,T}$, with the same enlarged family as there.
Put $a=4(1-\sigma)$ and $b=(1-\sigma)/(2-\sigma)$.
Lemma~\ref{lem:twisted-density} and \eqref{eq:band-density} give
\begin{equation}\label{eq:two-family-bounds}
 V_{d,T}\ll_\nu K^\nu
       \min\{((K+T^2)T^4)^b,K^a\}.
\end{equation}
If $T\ge K^{1/2}$, the weighted contribution is bounded by
\[
 \ll_\nu K^\nu\min\{T^{6b-1/2},K^aT^{-1/2}\}.
\]
The first expression increases with $T$ and the second decreases.
They meet at
\[
 T=K^{a/(6b)}=K^{2(2-\sigma)/3},
\]
whose exponent belongs to $[3/4,410/501]\subset[1/2,1]$.
Their common value is
\[
 K^{a-a/(12b)}=K^{(10-11\sigma)/3}=K^{C(\sigma)}.
\]
If $1\le T\le K^{1/2}$, the first bound in
\eqref{eq:two-family-bounds} gives instead
\[
 T^{-1/2}V_{d,T}\ll_\nu K^{b+\nu}T^{4b-1/2}
 \le K^{\max(b,\,3b-1/4)+\nu}.
\]
Both exponents here are less than $C(\sigma)$ throughout
$[\sigma_0,7/8]$. To check this explicitly, the functions
$C-b$ and $C-3b+1/4$ are decreasing, because
$b'(\sigma)=-(2-\sigma)^{-2}$. At $\sigma=7/8$ their values
are $1/72$ and $1/24$, respectively.
Sum over the $O(\log K)$ blocks and over at most $\tau(n)$ divisors
$d$. Absorb these factors into $X^\nu$ to obtain
\eqref{eq:restricted-high}.
\end{proof}

\section{Two estimates for one group of zeros}

Fix a group $I$, and temporarily abbreviate $a=a_I$, $h=h_I$.
For $1<K\le P$, let
\[
 B_{0,I}(n)=\sum_{q\le K}\sum_{b\bmod q}^{*}
 \int_{|\beta|\le1/(qQ)}
 g_q(b)W(\beta)A_{q,I}(b,\beta)
                  \e\bigl(-n(b/q+\beta)\bigr)\,d\beta.
\]
Define $B_{1,I}(n)$ by the same formula over $K<q\le P$.

\begin{lemma}\label{lem:two-estimates}
For every $\nu>0$, uniformly for integers $X<n\le2X$,
\begin{equation}\label{eq:small-band}
 |B_{0,I}(n)|\ll_\nu X^{h-1/2+\nu}K^{3a/4}.
\end{equation}
Moreover,
\begin{equation}\label{eq:large-band}
 \sum_{n\in\Z}|B_{1,I}(n)|^2
 \ll_\nu X^{2h+\nu}
 \begin{cases}
 P^{a-1},&a\ge1,\\
 K^{a-1},&a<1.
 \end{cases}
\end{equation}
\end{lemma}
\begin{proof}
In the small-denominator term, substitute the character expansion
of $A_{q,I}$ before taking absolute values. For fixed $q$, the
integration interval and $E_{\chi^*,I}(\beta)$ do not depend on
the reduced residue $b$. We may therefore apply
Lemma~\ref{lem:complete} to the complete residue sum, retaining
its additive phase.
By Lemma~\ref{lem:support}, only characters induced by primitive
characters admissible for $n$ can contribute.

For $q=rs$, the induced Gauss sum vanishes unless $s$ is squarefree
and $(r,s)=1$. In the other cases its magnitude is $\sqrt r$.
Thus the coefficient of the primitive character is bounded by
\[
 \frac{2\sqrt r}{\ph(rs)}\sqrt{(rs,n)}
 \ll_\nu X^\nu\sqrt{\frac{(r,n)}r}\frac{\sqrt{(s,n)}}s,
\]
where we used $(rs,n)=(r,n)(s,n)$ and
$1/\ph(q)\ll_\nu X^\nu/q$ for $q\le P$.
The band estimate \eqref{eq:band-linear} and the integral bound
\[
 \int_\R |W(\beta)|(1+X|\beta|)^{10}\,d\beta\ll X^{-1/2}
\]
therefore give
\[
 |B_{0,I}(n)|\ll_\nu X^{h-1/2+\nu}
 \sum_{r\le K}\sqrt{\frac{(r,n)}r}
 \sum_{\substack{\psi\bmod r\ \mathrm{primitive}\\
                 \psi\ \mathrm{admissible\ for}\ n}}Z_{\psi,I}
               \sum_{s\le K/r}\frac{\sqrt{(s,n)}}s.
\]
Dropping the squarefree and coprimality restrictions on $s$
increases this upper bound. After Lemma~\ref{lem:divisor}, we have
the useful general bound
\begin{equation}\label{eq:small-reduction}
 |B_{0,I}(n)|\ll_\nu X^{h-1/2+\nu}
 \sum_{r\le K}\sqrt{\frac{(r,n)}r}
 \sum_{\substack{\psi\bmod r\ \mathrm{primitive}\\
                 \psi\ \mathrm{admissible\ for}\ n}}Z_{\psi,I}.
\end{equation}
Apply Lemma~\ref{lem:restricted} and rename the arbitrarily small
losses to prove \eqref{eq:small-band}. The bound is uniform for
every integer $X<n\le2X$; no further exceptional integers have
been excluded. For upper groups we will apply
Lemma~\ref{lem:restricted-high} to \eqref{eq:small-reduction} instead.

For the other estimate, disjointness of the arcs, Parseval,
\eqref{eq:gauss}, and \eqref{eq:orthog} give
\[
 \sum_n|B_{1,I}(n)|^2
 \ll\sum_{K<q\le P}\frac1{q\ph(q)}
 \sum_{\chi\bmod q}|\tau(\chi)|^2
 \int_\R |W(\beta)|^2|E_{\chi^*,I}(\beta)|^2\,d\beta.
\]
For a fixed primitive conductor $r$, its coefficient is bounded by
\begin{align}
 \sum_{\substack{q=rs>K\\q\le P\\(r,s)=1}}
       \frac{r\mu^2(s)}{q\ph(q)}
 &\le\frac1{\ph(r)}\sum_{s>K/r}\frac1{s\ph(s)}\notag\\
 &\ll\frac{\min(1,r/K)}{\ph(r)}.\label{eq:tail-weight}
\end{align}
For completeness,
\[
 \frac{s}{\ph(s)}=\sum_{d\mid s}\frac{\mu^2(d)}{\ph(d)},
 \qquad \sum_{s\le y}\frac{s}{\ph(s)}\ll y;
\]
partial summation gives the tail estimate in \eqref{eq:tail-weight}.
For $K/r<1$, convergence of the full series gives the other part
of the minimum.

Apply \eqref{eq:band-integral} and absorb logarithms and the
factor $r/\ph(r)$ into $X^\nu$. This yields
\[
 \sum_n|B_{1,I}(n)|^2\ll_\nu X^{2h+\nu}
      \sum_{r\le P}\frac{\min(1,r/K)}r Z_{r,I}.
\]
On a dyadic block of size $R$, \eqref{eq:band-density} bounds
the last sum by $\ll_\nu X^\nu\min(1,R/K)R^{a-1}$.
Summing the blocks proves \eqref{eq:large-band}, after renaming
small losses. The conductor $r=1$ is included in the first block.
At $a=1$, absorb the additional logarithm into $X^\nu$.
\end{proof}

Let $R_D(n)$ be the major-arc Fourier contribution obtained by
replacing $A_q$ by $D_q$ over all $q\le P$.
By \eqref{eq:gauss}, \eqref{eq:D-trace}, and Parseval,
\begin{equation}\label{eq:remainder}
 \sum_n|R_D(n)|^2\ll P\log^C X.
\end{equation}
Indeed, the factor $1/q$ from the square sum cancels the factor
$q$ in the trace bound, and the integral of
$|W(\beta)|^2(1+X|\beta|)^{20}$ is bounded.

\section{Absorption and the exponent calculation}

Recall $P=X^{9/20}$ and put $p=9/20$.
For each lower group choose
\begin{equation}\label{eq:K}
 k_I=\frac{4(1-h_I)}{3a_I}-\delta,\qquad K_I=X^{k_I}.
\end{equation}
For the lower bands the limiting cutoff exponent is
\[
 \frac{4(1-\sigma)}{3a(\sigma)}=
 \begin{cases}
 2(2-\sigma)/9,&1/2\le\sigma\le13/20,\\
 3/10,&13/20\le\sigma<\sigma_0.
 \end{cases}
\]
It lies in $[3/10,1/3]$. Replacing $\sigma$ in the numerator by
$h_I$ reduces it by at most $4\omega/3$, since $a_I>1$.
Thus $1<K_I<P$ when $\omega,\delta<1/100$.
For $I_-$, take $k_{I_-}=1/3-\delta$.
Equation \eqref{eq:small-band} now gives
\begin{equation}\label{eq:low-absorb}
 |B_{0,I}(n)|\ll_\nu
 X^{1/2-(3/4)\delta a_I+\nu}
 \qquad(I\text{ a lower group}).
\end{equation}

For an upper group let $B_I(n)$ denote its entire major-arc
contribution over $q\le P$. Apply \eqref{eq:small-reduction}
with $K=P$, and then Lemma~\ref{lem:restricted-high}. This gives
\[
 |B_I(n)|\ll_\nu X^{h_I-1/2+\nu}P^{C(\sigma_I)}.
\]
The limiting exponent is
\begin{align*}
 \sigma-\frac12+\frac9{20}C(\sigma)
 &=\sigma-\frac12+\frac3{20}(10-11\sigma)\\
 &=1-\frac{13\sigma}{20}
 \le\frac12-c_0,\qquad
 c_0=\frac7{3340}>0.
\end{align*}
The maximum occurs at $\sigma=\sigma_0=129/167$.
Accounting for the band width, we have
\begin{equation}\label{eq:high-absorb}
 |B_I(n)|\ll_\nu X^{1/2-c_0+\omega+\nu}
 \qquad(I\text{ an upper group}).
\end{equation}
Choose $\nu<\delta/8$ and $\omega+\nu<c_0/2$.
Equations \eqref{eq:low-absorb} and \eqref{eq:high-absorb}, summed
over the fixed finite number of groups, imply
\begin{equation}\label{eq:small-total}
 B_{\rm small}(n):=
 \sum_{I\,\mathrm{lower}} B_{0,I}(n)
 +\sum_{I\,\mathrm{upper}}B_I(n)
 =o(\sqrt X/\log X)
\end{equation}
uniformly for all integers $X<n\le2X$.
The notation $B_{\rm small}$ refers to the size of the resulting
contribution; it includes all denominators for the upper groups.

Only the terms $B_{1,I}$ from lower groups remain for the mean-square
estimate. Since $a_I>1$, \eqref{eq:large-band}, after division by
$X$, has exponent
\[
 e_I=2h_I-1+\frac9{20}(a_I-1).
\]
For a lower band, put
\[
 F(\sigma)=2\sigma-1+\frac9{20}(a(\sigma)-1).
\]
Then $e_I\le F(\sigma_I)+2\omega$.
For $1/2\le\sigma\le13/20$,
\[
 F'(\sigma)=2-\frac{27}{10(2-\sigma)^2}>0,
 \qquad
 F(13/20)=\frac3{10}+\frac9{20}\left(\frac{14}{9}-1\right)
        =\frac{11}{20}.
\]
For $13/20\le\sigma<\sigma_0$,
\[
 F(\sigma)=2\sigma-1+
       \frac9{20}\left(\frac{40}{9}(1-\sigma)-1\right)
       =\frac{11}{20}.
\]
The additional group $I_-$ gives $e_{I_-}=9/20$.
Consequently, if
$B_{\rm large}(n)=\sum_{I\,\mathrm{lower}}B_{1,I}(n)$, then
Cauchy's inequality over the finite set of lower groups yields
\begin{equation}\label{eq:large-total}
 \sum_n|B_{\rm large}(n)|^2
 \ll_{\omega,\delta,\nu}X^{31/20+2\omega+\nu}.
\end{equation}
The positive margin $c_0$ in \eqref{eq:high-absorb} is fixed,
whereas the loss $2\omega+\nu$ can be made arbitrarily small.

\section{The principal term and its positivity}

The coefficient of the principal integral at denominator $q$ is
\[
 c_q(n)=\frac{\mu(q)}{q\ph(q)}
              \sum_{a\bmod q}^{*}G(a,q)\e(-an/q).
\]
For an odd prime $\ell$, the number of roots of $b^2\equiv n$
modulo $\ell$ is $1+(n/\ell)$. Summing the Ramanujan sum over $b$
therefore gives $\ell(n/\ell)$. Multiplicativity gives
\[
 c_q(n)=\frac{\mu(q)}{\ph(q)}\left(\frac nq\right)
 \quad(q\text{ odd}).
\]
For $q\equiv2\pmod4$, \eqref{eq:gauss} gives $c_q(n)=0$;
for $4\mid q$, $\mu(q)=0$. Define
\begin{equation}\label{eq:SP}
 \SSS_P(n)=\sum_{\substack{q\le P\\q\text{ odd}}}
                 \frac{\mu(q)}{\ph(q)}\left(\frac nq\right).
\end{equation}
At the ends of the major arcs, $X|\beta|\ge X^\eta$.
Thus \eqref{eq:VW} permits extension of the principal integrals
to $\R$ with an error $O_A(X^{-A})$; summing absolute coefficients
costs at most $O(\sqrt P)$.
Equations \eqref{eq:J} and \eqref{eq:theta-major} now give
\begin{equation}\label{eq:decomposition}
 \begin{aligned}
 R_X(n)={}&\kappa\sqrt X\,\SSS_P(n)+B_{\rm small}(n)+B_{\rm large}(n)\\
          &+R_D(n)+R_\m(n)+O(X^{-10}),
 \end{aligned}
\end{equation}
where $R_\m(n)=\int_\m S\Theta\e(-n\alpha)\,d\alpha$.
The Poisson errors are negligible here by $|S(\alpha)|\ll X\log X$.

Exclude temporarily the $O(\sqrt X)$ square integers $n$.
For nonsquare $n$, let $\chi_{4n}=(4n/\cdot)$ be the Kronecker
character, a real nonprincipal character modulo $4n$, possibly
imprimitive. It vanishes at $2$ and satisfies
$\chi_{4n}(q)=(n/q)$ for odd $q$. Set, for $\Ree z>0$,
\[
 F_n(z)=\sum_{q\text{ odd}}
        \frac{\mu(q)\chi_{4n}(q)}{\ph(q)q^z}.
\]
Euler factorization gives
\begin{equation}\label{eq:F}
 F_n(z)=\frac{H_n(z)}{L(1+z,\chi_{4n})},\qquad
 H_n(z)=\prod_{\ell>2}
 \frac{1-\chi_{4n}(\ell)\ell^{-z}/(\ell-1)}
      {1-\chi_{4n}(\ell)\ell^{-1-z}}.
\end{equation}
For every fixed $c<1$, each factor of $H_n$ is
$1+O_c(\ell^{-2+c})$ uniformly for $\Ree z\ge-c$.
It follows that $H_n$ is analytic and uniformly bounded there.
Under \eqref{hyp}, $F_n$ is analytic for $\Ree z>-1/8$.
Imprimitive Euler factors have zeros only on $\Ree s=0$, so do
not change this assertion.

Uniformly on $-1/16\le\Ree z\le1$,
\begin{equation}\label{eq:F-growth}
 |F_n(z)|\ll\bigl(n(2+|\Im z|)\bigr)^C
\end{equation}
for an absolute $C$. One way to obtain the reciprocal-$L$ bound
is to apply Borel--Carath\'eodory to a branch of $\log L(s,\chi_{4n})$
on disks centered at $2+it$. Take outer radius $35/32$ and inner
radius $34/32$. The outer disk stays strictly to the right of
$7/8$, while the inner disk reaches $15/16$. The Euler product
bounds $\log L$ at the center, and standard polynomial growth
bounds $\log|L|$ on the outer disk. This yields the asserted
polynomial bound for $1/L$ and hence \eqref{eq:F-growth}.

Define the completed singular series by
\begin{equation}\label{eq:SSS}
 \SSS(n)=F_n(0)=H_n(0)/L(1,\chi_{4n}).
\end{equation}
The local factors of $H_n(0)$ for character values $1,-1,0$ are
respectively
\[
 1-\frac1{(\ell-1)^2},\qquad (1-\ell^{-2})^{-1},\qquad1.
\]
Consequently $H_n(0)\asymp1$ with absolute positive constants.
Also $L(1,\chi_{4n})>0$ and $L(1,\chi_{4n})\ll\log(2n)$.
For the upper bound, truncate its Dirichlet series at $4n$ and
use periodic character sums of magnitude at most $4n$ for the tail.
Thus
\begin{equation}\label{eq:positive}
 \SSS(n)\gg1/\log X\qquad(X<n\le2X,\ n\text{ nonsquare}).
\end{equation}

\section{Completion by the quadratic large sieve}

\begin{lemma}\label{lem:completion}
Under \eqref{hyp}, for every $\nu>0$,
\begin{equation}\label{eq:completion}
 \sum_{\substack{X<n\le2X\\n\text{ nonsquare}}}
 |\SSS(n)-\SSS_P(n)|^2
 \ll_\nu X^\nu\left(X/P+\sqrt X\right).
\end{equation}
\end{lemma}
\begin{proof}
Let $L=X^A$ with $A$ a sufficiently large fixed constant.
Mellin inversion and a contour shift to $\Ree z=-1/16$, using
\eqref{eq:F-growth} and the exponential decay of $\Gamma(z)$, give
\[
 \sum_{q\text{ odd}}\frac{\mu(q)\chi_{4n}(q)}{\ph(q)}e^{-q/L}
 =\frac1{2\pi i}\int_{(1)}F_n(z)\Gamma(z)L^z\,dz
 =\SSS(n)+O(X^{-10}).
\]
The truncation at $q\le LX$ has exponentially small error.
Replacing $e^{-q/L}$ by $1$ for $q\le P$ costs
$\ll_\theta P^{1+\theta}/L$, which is also negligible after
increasing $A$. Hence, uniformly for the nonsquare $n$ in question,
\begin{equation}\label{eq:finite-tail}
 \SSS(n)-\SSS_P(n)=
 \sum_{\substack{P<q\le LX\\q\text{ odd}}}
       \frac{\mu(q)e^{-q/L}}{\ph(q)}\left(\frac nq\right)
       +O(X^{-10}).
\end{equation}

Consider a dyadic block $U<q\le2U$ of this finite sum.
Write uniquely
\[
 n=2^e d s^2,\qquad e\in\{0,1\},\quad d\text{ odd and squarefree},
 \quad s\ge1.
\]
Thus $s\le\sqrt{2X}$ and $d\le2X/(2^es^2)$.
For $c\in\{1,3\}$ restrict first to $d\equiv c\pmod4$.
Quadratic reciprocity, for odd squarefree $q$, gives
\[
 \left(\frac nq\right)=
 \left(\frac2q\right)^e
 (-1)^{((c-1)/2)((q-1)/2)}
 \left(\frac qd\right)\boldsymbol1_{(q,s)=1}.
\]
For fixed $e,c,s$, incorporate all factors other than $(q/d)$ into
\[
 b_q=\frac{\mu(q)e^{-q/L}}{\ph(q)}
 \boldsymbol1_{U<q\le2U}\boldsymbol1_{q\le LX}
 \boldsymbol1_{(q,s)=1}
 \left(\frac2q\right)^e(-1)^{((c-1)/2)((q-1)/2)}.
\]
Then $\sum_q|b_q|^2\ll_\theta U^{-1+\theta}$.
Apply \eqref{eq:HB} with $M=2X/(2^es^2)$ and length $2U$,
enlarging the sum over $d$ to all positive odd squarefree $d\le M$.
Summing over the two values of $e$, the two residue classes $c$,
and all $s\le\sqrt{2X}$ bounds the block mean square by
\[
 \ll_\theta (XU)^\theta U^\theta
 \sum_{s\le\sqrt{2X}}\left(\frac{X}{s^2U}+1\right)
 \ll_\theta (XU)^\theta U^\theta
                 \left(\frac XU+\sqrt X\right).
\]
Enlargement may include square $n$, which is harmless for this
upper bound on the finite block. Since $U\le X^{A+1}$ with $A$
fixed, choose $\theta$ sufficiently small relative to $\nu$.
Cauchy's inequality over the $O_A(\log X)$ blocks in
\eqref{eq:finite-tail}, followed by summation of $X/U$, proves
\eqref{eq:completion}. This argument includes nonsquare integers
with arbitrarily large square factors.
\end{proof}

\section{Completion of the exceptional-set bound}

For nonsquare $n$ in $X<n\le2X$, put
$M_X(n)=\kappa\sqrt X\,\SSS(n)$.
Combining \eqref{eq:decomposition}, \eqref{eq:minor},
\eqref{eq:large-total}, \eqref{eq:remainder}, and
Lemma~\ref{lem:completion}, we obtain
\begin{align}
 &\sum_{\substack{X<n\le2X\\n\text{ nonsquare}}}
 |R_X(n)-M_X(n)-B_{\rm small}(n)|^2\notag\\
 &\quad\ll_{\omega,\delta,\nu}
 X^{31/20+2\omega+\nu}.
 \label{eq:final-ms}
\end{align}
The minor-arc term is $X^{2+\nu}/P=X^{31/20+\nu}$.
The singular-series completion contributes
$\ll X^\nu(X^2/P+X^{3/2})$, because its mean square is multiplied
by $X$. The elementary remainder contributes only $P\log^CX$.
All are bounded by the right side of \eqref{eq:final-ms}.

Let $R_{\rm pp}(n)$ be the nonnegative contribution to $R_X(n)$
from $n-k^2=p^j$ with $j\ge2$. Then
\begin{equation}\label{eq:pp}
 \sum_{X<n\le2X}R_{\rm pp}(n)
 \ll\sqrt X\sum_{\substack{p^j\le3X\\j\ge2}}\log p
 \ll X\log^2X.
\end{equation}
By \eqref{eq:positive}, there is $c>0$ such that
$M_X(n)\ge c\sqrt X/\log X$ throughout the nonsquare integers
in the interval. By \eqref{eq:small-total}, with the parameter restrictions above,
for sufficiently large $X$,
\[
 |B_{\rm small}(n)|\le c\sqrt X/(4\log X)
\]
uniformly in $n$. Equation \eqref{eq:pp} implies that at most
$O(\sqrt X\log^3X)$ integers satisfy
$R_{\rm pp}(n)>c\sqrt X/(4\log X)$.

If $n$ has no representation as a prime plus an integer square,
then $R_X(n)=R_{\rm pp}(n)$. Outside that prime-power exceptional
set and the $O(\sqrt X)$ square integers, it follows that
\[
 |R_X(n)-M_X(n)-B_{\rm small}(n)|
                      \ge c\sqrt X/(2\log X).
\]
Divide \eqref{eq:final-ms} by a constant multiple of
$X/\log^2X$. We obtain
\[
 \#\{X<n\le2X:n\ne p+k^2\}
 \ll_{\omega,\delta,\nu}
 X^{11/20+2\omega+\nu}\log^2X+\sqrt X\log^3X.
\]
Given $\eps>0$, choose
\[
 0<\omega<\min(1/100,\eps/16,c_0/4),\qquad
 0<\delta<1/100,
\]
and then choose $0<\nu<\min(\delta/8,\eps/8,c_0/4)$.
These choices ensure both the uniform absorption and a total
exponent loss smaller than $\eps$, after absorbing logarithms.
We obtain
\[
 \#\{X<n\le2X:n\ne p+k^2\}\ll_\eps X^{11/20+\eps}.
\]
The constants, including the finite number of bands, now depend
only on $\eps$ and the fixed weights. Summing over dyadic intervals
proves Theorem~\ref{thm:main}.

\begin{remark}
The cutoff $P=X^{9/20}$ makes the minor-arc and singular-series
completion errors have mean-square size $X^{31/20+\nu}$.
The lower zero bands have the same limiting scale. The upper bands
are uniformly smaller than the positive main term, by the new
restricted-family estimate and the explicit margin $7/3340$.
The excluded squares and proper prime powers cost only
$O(\sqrt X\log^3X)$ integers. No optimality of $11/20$ is asserted.
\end{remark}

\section{AI-use declaration and contributions}\label{sec:ai}

ChatGPT, an AI system provided by OpenAI, was used in developing this
work. Its contribution included helping formulate and check the
conditional exceptional-set argument, working through the analytic
estimates and parameter choices, checking algebra, locating and
comparing source statements, and preparing the \LaTeX\ source.

The human authors retain responsibility for the mathematical content,
the use and interpretation of the cited results, and the final text.
The arguments in this manuscript have not been formalized in Lean.

\section*{Authors' addresses}

\noindent\begin{minipage}{\textwidth}
\small
\textbf{J\'anos Pintz}\\
HUN-REN Alfr\'ed R\'enyi Institute of Mathematics\\
Re\'altanoda u.~13--15, H-1053 Budapest, Hungary\\
E-mail: \href{mailto:pintz@renyi.hu}{\nolinkurl{pintz@renyi.hu}}
\end{minipage}

\medskip
\noindent\begin{minipage}{\textwidth}
\small
\textbf{Ingo Alth\"ofer}\\
Friedrich-Schiller University Jena\\
Faculty of Mathematics and Computer Science\\
D-07740 Jena, Germany\\
E-mail: \href{mailto:ingo.althoefer@uni-jena.de}{\nolinkurl{ingo.althoefer@uni-jena.de}}
\end{minipage}
\end{document}